\documentclass[alpha-refs]{wiley-article}

\usepackage{graphicx}
\usepackage[space]{grffile}
\usepackage{latexsym}
\usepackage{textcomp}
\usepackage{longtable}
\usepackage{tabulary}
\usepackage{booktabs,array,multirow}
\usepackage{amsfonts,amsmath,amssymb, enumitem}
\usepackage{natbib}
\usepackage{url}
\usepackage{hyperref}
\hypersetup{
   colorlinks = true,
   citecolor = blue,
   urlcolor = purple,
   linkcolor = red,
   ,pdfborder={0 0 0}
}

\usepackage{etoolbox}
\newif\iflatexml\latexmlfalse

\AtBeginDocument{\DeclareGraphicsExtensions{.pdf,.PDF,.eps,.EPS,.png,.PNG,.tif,.TIF,.jpg,.JPG,.jpeg,.JPEG}}

\usepackage[utf8]{inputenc}
\usepackage[english]{babel}

\usepackage{siunitx}

\iflatexml

\else

\papertype{}

\title{Asymptotically Optimal Sequential Multiple Testing Procedures under Dependence}

\author[1]{Monitirtha Dey}
\author[2]{Subir Kumar Bhandari}

\affil[1]{Institute for Statistics, University of Bremen, Bremen, Germany. Email: mdey@uni-bremen.de}
\affil[2]{Interdisciplinary Statistical Research Unit, Indian Statistical Institute, Kolkata, India. Email: subirkumar.bhandari@gmail.com}

\runningauthor{Monitirtha Dey, Subir Kumar Bhandari}

\begin{document}

 \maketitle
\begin{abstract}
This work focuses on the multiple testing problem in a sequential framework where the observations in various streams are dependent, from an optimality viewpoint. We explore this issue by considering the classical means-testing problem in an equicorrelated Gaussian and sequential framework. We focus on sequential test procedures that control the type I and type II familywise error probabilities at pre-specified levels. We establish that our proposed test procedures achieve the optimal expected sample sizes under every possible signal configuration asymptotically, as the two error probabilities vanish at arbitrary rates. Towards this, we elucidate that the ratio of the expected sample size of our proposed rule and that of the classical SPRT (both in the respective context and with the same type I and type II errors) goes to one asymptotically. As a byproduct, our technique provides an alternating proof of existing results on optimal sequential test procedures, extending those to the equicorrelated case. Generalizing this, we show that our proposed procedures, with appropriately adjusted critical values, are asymptotically optimal for controlling any multiple testing error metric lying between multiples of FWER in a certain sense. This class of metrics also includes FDR/FNR and pFDR/pFNR.

\noindent \textbf{Keywords.} ~\emph{Multiple testing under dependence}, \emph{Sequential multiple testing}, \emph{Sequential analysis},
\emph{Asymptotic Optimality}, \emph{Correlated Gaussian.}
\end{abstract}%

\section{Introduction\label{sec:chap6sec1}}
Simultaneous statistical inference has been a cornerstone in the statistics methodology literature, particularly because of its fundamental theory and paramount applications. The mainstream multiple testing literature has traditionally considered two frameworks:
\begin{enumerate}
    \item The sample size is deterministic, i.e., the full data is available while testing. The classical Bonferroni method or the multiple testing procedures (MTPs henceforth) proposed by \cite{BH}, \cite{Holm1979}, \cite{Hochberg1988}, \cite{Hommel1988} - each of them is valid under the fixed sample size paradigm.
    \item The test statistics corresponding to the various tests are independent. For example, the BH procedure was initially shown to provide valid control of FDR for independent test statistics \citep{BH}. The Sidak's procedure holds for independent test statistics and for test statistics with certain parametric distributions.
\end{enumerate}

However, in many modern applications, these assumptions are routinely violated:
\begin{enumerate}
    \item Quite often, the data is streaming or arriving sequentially. In such scenarios, a statistician would need to draw inferences each time new data arrives based on the data available till that time point. In multiple-endpoint clinical trials, patients are collected sequentially or in groups. At each interim stage, the researcher has to decide whether to accept or reject or to collect more samples. Of late, MTPs that can handle these kinds of sequential data have been proposed and studied. 

    When testing multiple hypotheses simultaneously with data collected from a different stream for each hypothesis, one might consider the natural generalization of Wald's sequential framework where all the data streams are terminated simultaneously. This setup finds application in multiple access wireless
network \citep{Rappaport} and multisensor surveillance systems \citep{RFV}. The last decade has witnessed significant progress on this line of work: \cite{DeBaronJSPI, DeBaronSEQ, DeBaron}, \cite{BartroffSong2014, BartroffSong2016}, \cite{SongFellouris2017, SongFellouris2019}, \cite{heBartroff2021}, \cite{roy2023largescale}.
    \item Simultaneous inference problems arising in many disciplines often involve correlated observations \citep{HandbookMCP}. For example, in microRNA expression data, several genes may cluster into groups through their transcription processes and exhibit high correlations. Functional magnetic resonance imaging (fMRI) studies and multistage clinical trials also concern dependent observations. 
    
    Simultaneous testing methods under dependence have been studied by ~\cite{SunCai2009}, ~\cite{efron2007}, ~\cite{LZP}, among others. \cite{Qiu2005} demonstrated that many FDR controlling procedures lose power significantly under dependence. \cite{HuangHsu} mention that stepwise decision rules based on modeling of the dependence structure are generally superior to their counterparts that do not consider the correlation. 
    
    Recently, there has been a flurry of work on multiple hypotheses testing in connection with signal detection \citep{shiterburd2022ksample}. The loss is usually just the total number of mistakes, i.e., the Hamming loss. \cite{Butucea} consider the minimax risk of variable selection under expected Hamming loss in the (correlated) Gaussian mean model in $\mathbb{R}^d$.
\end{enumerate}

However, there is comparatively less literature which studies the multiple testing problem in a sequential framework where the test statistics corresponding to the various streams are dependent. The works by \cite{DeBaronJSPI, DeBaronSEQ, DeBaron} and \cite{BartroffSong2014, BartroffSong2016}  consider arbitrary between-stream correlation. \cite{DeBaronSEQ} derived asymptotically optimal procedures  for controlling type I and type II familywise error under Pitman alternative.

In this paper we consider the classical means-testing problem in a equicorrelated Gaussian and sequential framework. We focus on sequential test procedures that control the type I and type II familywise error probabilities at pre-specified levels. We establish that our proposed test procedures achieve the optimal expected sample sizes under every possible signal configuration asymptotically, as the two error probabilities vanish at arbitrary rates. Table \ref{results} summarizes the relevant asymptotic optimality results in sequential multiple testing in different correlated normal setups.

\begin{table}[ht]
\centering
\normalsize\begin{tabulary}{1.0\textwidth}{CCCC}
\hline
\textbf{Error rate under consideration} & \textbf{Prior info on number of signals} & \textbf{Existing optimality results under independence} & \textbf{New optimality results under dependence}\\ 
 $(FWER_{I}, FWER_{II})$ & known exactly& Theorem 5.3 of \citep{SongFellouris2017}  & \autoref{1stoptimality}\\
   $(FWER_{I}, FWER_{II})$ & lower and upper bounds known & Theorem 5.5 of \citep{SongFellouris2017} & \autoref{2ndoptimality}\\
  General error rate & known exactly & Theorem 3.1 of \cite{heBartroff2021} & \autoref{theoremongenerror}\\
   General error rate & lower and upper bounds known & Theorem 4.1 of \cite{heBartroff2021} & \autoref{theoremongenerror2} \\\hline
 \end{tabulary}
 \caption{Asymptotic optimality results in sequential multiple testing literature under correlated normal setup (the theorem numbers, unless mentioned otherwise, are as in this paper)}
 {\label{results}}
\end{table}

The equicorrelated setup \citep{Finner2001a, FDR2007, Sarkar2007, Cohen2009, Delattre, Proschan, das_2021, DasBhandari2025, deybhandaristpa, dey2024beyond, deycstm, royspl}  characterizes the exchangeable situation and also covers the problem of comparing a control against several treatments. Simultaneous testing of normal means under equicorrelated frameworks has attracted considerable attention in recent years. \cite{das_2021} showed that the FWER of the Bonferroni procedure is asymptotically bounded above by $\alpha(1-\rho)$ where $\alpha$ is the desired level and $\rho$ is the common correlation. \cite{deyspl} proved that the Bonferroni FWER asymptotically approaches zero. \cite{deystpa} extended this to any step-down procedure and some other test procedures. 

\vspace{3mm}

The \textbf{main contributions of this work} are as follows:

\begin{enumerate}
\item We provide alternative proofs of the results of \cite{SongFellouris2017}, which in turn cover the equicorrelated set-up also. 
\item We generalize the results of \cite{heBartroff2021} to the equicorrelated setup.
\item Our results illuminate that dependence among test-statistics might be a blessing or a curse, subject to the type of dependence or the underlying paradigm (i.e, fixed sample-size or sequential).
\item  We also extend our results to the case when both the signal and the equicorrelation are unknown. We would like to mention that we could not find any existing relevant literature on the asymptotic ASN expression of classical SPRT (i.e, in the traditional one null vs one alternative scenario) when the signal strength is unknown. From this perspective, this is a welcome addition!
\end{enumerate}

This paper is structured as follows. We first formally introduce the framework with relevant notations and mention some existing results on the sequential test rules in the next section. We discuss the asymptotic expansion of expected sample size required by the SPRT in Section \ref{sec:chap6sec3}. We propose feasible sequential rules and establish their asymptotic optimality in Section \ref{sec:chap6sec4}. Section \ref{sec:chap6sec5} extends our asymptotic results to a wide class of error rates. Section \ref{sec:chap6sec6} considers the case when both the signal strength and the equicorrelation are unknown. Section \ref{sec:chap6sec7} presents a simulation study elucidating the performances our proposed tests empirically.
We conclude with a brief discussion in Section \ref{sec:chap6sec8}.

Throughout this work, $M_{K}(\rho)$ stands for the $K \times K$ matrix with $1$ at each diagonal position and $\rho \geq 0$ at each off-diagonal position. Also, as we shall discuss several hypothesis testing problems (HTPs henceforth) in this paper, we shall name each problem for notational convenience. 

\section{Preliminaries\label{sec:chap6sec2}} 
\subsection{The Testing Framework}
We consider $K \geq 2$ data streams:
\begin{align*}
    & X_{11}, X_{12}, \ldots, X_{1n}, \ldots \\
    & X_{21}, X_{22}, \ldots, X_{2n}, \ldots \\
    & \vdots \\
    & X_{K1}, X_{K2}, \ldots, X_{Kn}, \ldots .
\end{align*}
Here $X_{ij}$ denotes the $j$'th observation of the $i$'th data stream, $i \in [K] := \{1, \ldots, K\}$. We assume throughout the paper that the elements of a given stream are independently and identically distributed but not necessarily the streams are independent of each other. For each $i \in [K]$, we consider two simple hypotheses:
$$\textbf{HTP 1.} \hspace{2mm} H_{0i} : X_{ij} \sim N(0,1) \text{\hspace{2mm}for each\hspace{2mm}} j \in \mathbb{N} \quad vs \quad H_{1i} : X_{ij} \sim N(\mu_{i},1)  \text{\hspace{2mm} for each\hspace{2mm}} j \in \mathbb{N}$$
 where $\mu_{i} >0$. 
 
 We assume that for each $j \in \mathbb{N}$, $(X_{1j}, \ldots, X_{Kj})$ follows a multivariate normal distribution with variance covariance matrix $M_{K}(\rho)$ for some $\rho \geq 0$. We consider unit variances since the multiple testing literature often assumes that the variances are known \citep{Abramovich, Bogdan, das_2021, deyspl, Donoho}. We will say that there is “noise” in the $i$th stream if $H_{0i}$ is true and there is “signal” in the $i$th stream otherwise.

Let $\mathcal{S}_{n}^{i}$ denote the $\sigma$-field generated by the first $n$ observations of the
$i$ th data stream, i.e., $\sigma(X_{i1},\ldots,X_{in})$. Let $\mathcal{S}_n$ be the $\sigma$-field generated by the
first $n$ observations in all streams, that is, $\sigma(\mathcal{S}_{n}^{i},  i \in [K]$), where $n \in \mathbb{N}$. The data in all streams are observed sequentially and our goal is to terminate sampling as soon as possible, and upon stopping to solve the $K$ hypothesis testing problems subject to certain error control guarantees. 

We define a \textit{sequential multiple testing procedure} or a \textit{sequential test} as a pair $(T,d)$ where $T$ is an $\{\mathcal{S}_n\}$-stopping
time at which we stop sampling in all streams, and $d$ an $\mathcal{S}_T$-measurable, $K$-dimensional vector of Bernoulli random variables, $(d_1,\ldots,d_K)$, so that we select the alternative hypothesis in the $i$ th stream if and only if $d_i = 1$. Let $\mathcal{D}$ be the subset of streams in which the alternative hypothesis is
selected upon stopping, i.e., $\{i \in [K] : d_i = 1\}$. Suppose $\mathcal{A}$ is the true subset of indices for which the alternative hypothesis is true. For an arbitrary sequential test $(T, d)$ we have
$$
\begin{aligned}
& \{(\mathcal{D} \backslash \mathcal{A}) \neq \emptyset\}=\bigcup_{j \notin \mathcal{A}}\left\{d_{j}=1\right\},  \{(\mathcal{A} \backslash \mathcal{D}) \neq \emptyset\}=\bigcup_{k \in \mathcal{A}}\left\{d_{k}=0\right\} .
\end{aligned}
$$
For any subset $\mathcal{B} \subset[K]$ let $\mathbb{P}_{\mathcal{B}}$ be defined as the distribution of $\left\{\mathbf{X}_{n}, n \in \mathbb{N}\right\}$ when $\mathcal{B}$ is the true subset of signals. Thus, the two types of familywise error rates are given by 
\begin{align*}
    & FWER_{I, \mathcal{A}}(T,d) = \mathbb{P}_{\mathcal{A}}((\mathcal{D} \backslash \mathcal{A}) \neq \emptyset)= \mathbb{P}_{\mathcal{A}}\big[\text{$\textcolor{black}
    {(T,d)}$ makes at least one false rejection}\big],\\
& FWER_{II, \mathcal{A}}(T, d) = \mathbb{P}_{\mathcal{A}}((\mathcal{A} \backslash \mathcal{D}) \neq \emptyset) = \mathbb{P}_{\mathcal{A}}\big[\text{$\textcolor{black}{(T,d)}$ makes at least one false acceptance}\big].
\end{align*}
We write $FWER_{i, \mathcal{A}}(T,d)$ as $FWER_{i}(T,d)$ for $i = I, II$.
We are here concerned with sequential test procedures that control these two error probabilities below pre-specified precision levels $\alpha$ and $\beta$ respectively, where $\alpha, \beta \in(0,1)$, for any possible subset of signals. In order to be able to incorporate prior information as \cite{SongFellouris2017}, we assume that the true subset of signals is known to belong to a class $\mathcal{P}$ of subsets of $[K]$. We consider the class
$$\Delta_{\alpha, \beta}^{FWER}(\mathcal{P}):=\left\{(T, d): FWER_{I}(T,d) \leq \alpha \text { and } FWER_{II}(T,d) \leq \beta \text { for every } \mathcal{A} \in \mathcal{P}\right\}.$$
In particular, two general classes $\mathcal{P}$ will be considered. In the first case, it is known beforehand that there are exactly $m$ signals in the $K$ streams, where $1 \leq m \leq K-1$. In the second case, although the exact number of signals might not be known, strict lower and upper bounds for the same are available. These two cases correspond to the classes 
$$\mathcal{P}_{m}:=\{\mathcal{A} \subset[K]:|\mathcal{A}|=m, 0<m<K\}, \quad \mathcal{P}_{\ell, u}:=\{\mathcal{A} \subset[K]: 0 < \ell <|\mathcal{A}| < u <K\},$$
\noindent respectively. 

\textcolor{black}{
\begin{remark}
   In the previous works on sequential multiple testing under independence, the mean under the alternative hypothesis is not necessarily positive, neither the mean under the null  is necessarily zero. For example, in the setting of \cite{SongFellouris2017}, mean is $\theta_{0i}$ under the null and is $\theta_{1i}$ under the alternative, with $\theta_{1i} >\theta_{0i}$. Note that testing $\theta_{0i}$ vs $\theta_{1i}$ is equivalent to testing $0$ vs $\theta_{1i}-\theta_{0i}(>0)$. Therefore, in HTP 1, we are assuming zero mean under the null and $\mu_{i}>0$ under the alternative.
\end{remark}}
\subsection{Asymptotic Optimality for controlling FWER}
We are interested in finding sequential MTPs belonging to  $\Delta_{\alpha, \beta}^{FWER}(\mathcal{P}_m)$ or $\Delta_{\alpha, \beta}^{FWER}(\mathcal{P}_{\ell, u})$ which are optimal in the natural sense of Wald's sequential framework, i.e., which achieve the minimum possible expected sample size, under each possible signal configuration, for small error probabilities. 

\begin{definition}\citep{SongFellouris2017}
Let $\mathcal{P}$ be a given class of subsets and let $\left(T^{*}, d^{*}\right)$ be a sequential test that can \textcolor{black}{be} designed to belong to $\Delta_{\alpha, \beta}^{FWER}(\mathcal{P})$ for any given $\alpha, \beta \in(0,1)$. $\left(T^{*}, d^{*}\right)$ is called \textit{asymptotically optimal in the class $\mathcal{P}$ for controlling FWER}, if for every $\mathcal{A} \in \mathcal{P}$ we have,
$$\lim_{\alpha, \beta \to 0} \dfrac{\mathbb{E}_{\mathcal{A}}\left[T^{*}\right]}{\displaystyle \inf_{(T, d) \in \Delta_{\alpha, \beta}^{FWER}(\mathcal{P})} \mathbb{E}_{\mathcal{A}}[T]} = 1$$
where $\mathbb{E}_{\mathcal{A}}$ denotes the expectation under $\mathbb{P}_{\mathcal{A}}$. 
\end{definition}

We shall often write $x \sim y$ to mean $x / y \rightarrow 1$. We denote cardinality by $|\cdot|$ and, for any two real numbers $x, y$ we set $x \wedge y=\min \{x, y\}$ and $x \vee y=\max \{x, y\}$. 
This paper aims to propose feasible sequential test procedures that are asymptotically optimal in the classes $\mathcal{P}_{m}$ and $\mathcal{P}_{\ell, u}$. 

\cite{SongFellouris2017} consider independent data streams and propose sequential test procedures (which they referred to as the “gap” and “gap-intersection” procedures) that are asymptotically optimal in the classes $\mathcal{P}_{m}$ and $\mathcal{P}_{\ell, u}$ under independence.

\cite{heBartroff2021} showed that the “gap” and “gap-intersection” procedures proposed by \cite{SongFellouris2017}, with appropriately adjusted critical values, are asymptotically optimal for controlling any multiple testing error metric that is bounded between multiples of FWER in a certain sense. This class of metrics includes FDR/FNR but also pFDR/pFNR, the per-comparison and per-family error rates, and the false positive rate. We shall show in section 5 that our proposed procedures also are asymptotically optimal for controlling any multiple testing error metric that is bounded between multiples of FWER in a certain sense.

\section{Asymptotic Expansion of Expected Sample Size of One-sided SPRT\label{sec:chap6sec3}}

This work aims to propose feasible sequential test procedures that require minimum expected sample size in the classes $\mathcal{P}_{m}$ and $\mathcal{P}_{\ell, u}$ asymptotically as $\alpha, \beta \longrightarrow 0$. We shall establish this asymptotic optimality by comparing the ratios of the expected sample sizes of our proposed procedure and the one-sided sequential probability ratio test (SPRT) ((both in the respective context and with the same type I and type II errors). Therefore, we at first focus on the expected sample size (or average sample number (ASN)) of SPRT. 

\textcolor{black}{Consider the following hypothesis testing problem:
$$\textbf{HTP 2.} \hspace{2mm} H_0 : U_i \sim N(\theta_0, \sigma^2) \quad vs \quad H_1: U_i \sim N(\theta_1, \sigma^2), \quad \theta_0 <\theta_1.$$
Assume that if $N(\theta_1, \sigma^2)$ is the true density, each observation costs a fixed amount. Here we would wish to stop sampling as soon as possible and reject the null $H_0$.}

\textcolor{black}{ \cite{Siegmund} gives the following example on one-sided testing. Suppose a new drug is being marketed under the null hypothesis that its side effects are insignificant. However, doctors prescribing the drug must record and report the side effects. As long as the hypothesis $H_0 : U_i \sim N(\theta_0, \sigma^2)$ remains reasonable, no regulation is necessary. If it even appears that the level of side effects is unacceptably high ($H_1 : U_i \sim N(\theta_1, \sigma^2)$), the regulatory body will announce this and and withdraw the drug.}

\textcolor{black}{A \textit{test} of $H_0: U_i \sim N(\theta_0, \sigma^2)$ in the sense described above is provided by a stopping time $T$. If $T<\infty$, $H_0$ is rejected. We seek a stopping time for which $\mathbb{P}_{H_0}\{T<\infty\} \leq \alpha$ for some pre-specified small $\alpha$ and for which $\mathbb{E}_{H_1}(T)$ is minimum. At the $n$-th stage, the one-sided SPRT \citep{LaiSPRT, Siegmund} for this problem is as follows:}

\textcolor{black}{(a) reject $H_0$ if
$$
\sum_{i=1}^n U_i-\frac{n\left(\theta_1+\theta_0\right)}{2} \geqslant |\log \alpha| \cdot \frac{\sigma^2}{\theta_1-\theta_0} .
$$}

\textcolor{black}{(b) continue sampling otherwise.}

\textcolor{black}{We denote the above one-sided SPRT as $SPRT_{\alpha}(\theta_{0}, \theta_{1}, \sigma^2)$. We denote its stopping time as $T_{SPRT_{\alpha}(\theta_{0}, \theta_{1}, \sigma^2)}$.}

\textcolor{black}{The optimality of SPRT among all fixed-sample-size or sequential tests is well-known \citep{CRRao, Nitis} . We shall utilize the following result on the optimality of the one-sided SPRT:}
\textcolor{black}{
\begin{theorem}(\cite{Siegmund}, page 20) \label{SPRToptimality} 
Consider the class of all fixed-sample-size or sequential tests for the problem HTP 2, having the type I error probability not more than $\alpha<1$.
\begin{enumerate}
  \item For any stopping time $T$ in this class, 
$$
\mathbb{E}_{H_1}(T) \geq |\log \alpha| \cdot \frac{2 \sigma^2}{(\theta_1-\theta_0)^2}.
$$
  \item The approximate expression for the ASN of $SPRT_{\alpha}(\theta_0, \theta_1, \sigma^2)$ when $\alpha \to 0$ is
    $$\mathbb{E}_{H_{1}}[T_{SPRT_{\alpha}(\theta_{0}, \theta_{1}, \sigma^2)}] \sim |\log \alpha| \cdot \frac{2 \sigma^2}{(\theta_1-\theta_0)^2}.$$ Here $x \sim y$ implies $x/y \to 1$ as $\alpha \to 0$. 
\end{enumerate} 
In other words, the one-sided SPRT with type I error probability $\alpha$ minimizes $\mathbb{E}_{H_1}(T)$ in this class when $\alpha \to 0$.
\end{theorem}}
The following result will be crucial in proving our asymptotic optimality results. 
\begin{theorem}\label{sprtoptimalitynew}(see Theorem 2 of \cite{Mei})
    Let $L >2$ and $[L]:=\{1, \ldots, L\}$. Suppose that for each $i \in [L]$, we have a sequence of i.i.d. random variables $\left\{\epsilon_{in}, n \in \mathbb{N}\right\}$ having means $\mathbb{E}(\epsilon_{in}) = \theta_{i} >0$ and finite variances. Let
$$
M_{i,n}=\sum_{j=1}^n \epsilon_{ij}, \quad j \in \mathbb{N}. 
$$
Let $\left(a_1, \ldots, a_L\right)  \in \mathbb{R}^{L}$ be arbitrary. Consider the stopping time
$\mathcal{W}=\inf_{n \geq 1} \left\{ M_{i,n} \geq a_i \quad \text{for all} \quad i\in [L]\right\} .$
Then, as $a_1, \ldots, a_L \rightarrow \infty$,
$$
\mathbb{E}[\mathcal{W}] \leq \max _{i \in[L]}\left(\frac{a_i}{\theta_{i}}\right)+O\left(L \sqrt{\max _{i \in[L]}\left\{a_i\right\}}\right).
$$
\end{theorem}

\section{Main Results for FWER Controlling Tests\label{sec:chap6sec4}}
\subsection{A distributionally equivalent representation of the vector observations}

Let $\mathbf{X}_n = (X_{1n}, \ldots, X_{kn})$ denote the $K$-dimensional vector storing all the observations collected at time $n$. Suppose $\boldsymbol{\mu} = (\mu_1, \ldots, \mu_K)'$ where $\mu_{i}$ is 0 for $i \notin \mathcal{A}$ and strictly positive otherwise. Then, for each $n \geq 1$, we have
$\mathbf{X}_n \sim MVN_{K}(\mathbf{\boldsymbol{\mu}}, M_K(\rho)).$
This implies,  
\begin{equation*}
    \mathbf{X}_n \overset{d}{=} \mathbf{Z}_n + V_n \cdot \mathbf{1}_K
\end{equation*}
where $\mathbf{Z}_n = (Z_{1n}, \ldots, Z_{kn}) \sim MVN_{K}(\boldsymbol{\mu}, (1-\rho)\cdot I_{K})$ and $V_n \sim N(0,\rho)$ are independent. Here $\mathbf{1}_K$ denotes the $K$ dimensional vector of all ones. Thus, for each $i \in [K]$ and for each $j \geq 1$, 
$X_{ij} \overset{d}{=} Z_{ij} + V_j.$
This gives, 
\begin{align*}
      \sum_{j=1}^n X_{ij} \overset{d}{=} \sum_{j=1}^n Z_{ij} + \sum_{j=1}^n V_j 
    \implies & S_{i,n} \overset{d}{=} R_{i,n} + \sum_{j=1}^n V_j \quad \text{(say)}\\
    \implies & S_{i,n} -S_{i^{\prime},n} \overset{d}{=} R_{i,n} - R_{i^{\prime},n}.
\end{align*}
We also note that
$$S_{(m),n}-S_{(m+1),n} \overset{d}{=} R_{(m),n}-R_{(m+1),n}$$
where $S_{(1),n} \geq \cdots \geq S_{(K),n}$ and $R_{(1),n}\geq \cdots \geq R_{(K),n}$. This implies, although we can not directly observe $Z_{ij}$'s or $R_{i,n}$'s, we can observe the quantities $R_{i,n}-R_{i^{\prime},n}$ and $R_{(m),n}-R_{(m+1),n}$. The distributions of $V_{j}$ or $\sum_{j} V_j$ do not involve $\mu$. Therefore, inference on $\mu$ based on $R_{i,n}$ would convey the same amount of information as the inference based on $S_{i,n}$ would have. 

It is worth noting that the equicorrelated model also arises naturally in the common (but, important) treatments versus control setting.  Suppose the control $Y_0$ has a normal distribution with variance $\rho$  and each treatment $Y_i$ has a normal distribution with variance  $1-\rho$ and all these random variables are independent. Each test will be based on  $X_i = Y_i -Y_0$.  Therefore, the $X$’s have a multivariate Normal distribution with covariance matrix  $M_{K}(\rho)$. Thus $\rho$ and $1-\rho$ can be thought of as the accuracy of the control observations relative to the treatment observations. This might give an added intuition to the role of $\rho$. Also note that larger values of $\rho$ lead to larger correlations among the $X_{i}$'s. So one might expect that
it should require fewer observations on average to detect the signals for larger values of $\rho$. We shall observe in the subsequent sections that this is indeed the case.

 \subsection{Proposed procedure for known number of signals}
 \textcolor{black}{
\noindent When it is known beforehand that there are exactly $m$ signals, we have
$$\mu_{(1)} \geq \mu_{(2)}\geq \cdots \geq \mu_{(m)} > \mu_{(m+1)}=\cdots=\mu_{(K)}=0.$$}
Mimicking the gap rule proposed by \cite{SongFellouris2017}, we propose the following stopping time:
 $$\begin{aligned}
T_{Gap}^{\star}(m,c) := & \inf \left\{n \geq 1:  S_{(m),n} - S_{(m+1),n} \geq \frac{1-\rho}{\mu_{(m)}} \cdot c\right\}
= \inf \left\{n \geq 1:  R_{(m),n} - R
_{(m+1),n} \geq \frac{1-\rho}{\mu_{(m)}} \cdot c\right\}
\end{aligned}$$
where $c:=|\log (\alpha \wedge \beta)|+\log (m(K-m))$.
Let $G= (1-\rho)\cdot c/\mu_{(m)}$ be the r.h.s of the above inequality. 

\textcolor{black}{The problem is to find the indices having strictly positive $\mu$'s. We denote this classification problem as $\mathcal{C}$. Let us consider two other classification problems:}
\textcolor{black}{
\begin{enumerate}
    \item \textbf{Problem $\mathcal{C}_{1}$:} Consider $K$ data streams where it is known beforehand that there are exactly $m$ signals. Also, suppose that all the strictly positive $\mu$'s are equal to $\mu_{(m)}$ ($\mu_{(m)}$ is as defined in problem $\mathcal{C}$). So, essentially $m$ many $\mu$'s are equal to $\mu_{(m)}$ while the rest are zero. 
    \item \textbf{Problem $\mathcal{C}_{2}$:} Consider two independent data streams: for each $j \in \mathbb{N}$,
    $$Z_{i_{1}j} \sim N(\mu_1, 1-\rho) \quad \text{and} \quad Z_{i_{2}j} \sim N(\mu_2, 1-\rho).$$
    Suppose it is known beforehand that among these two streams there is exactly $1$ signal. More specifically, assume that exactly one of $\mu_1$ and $\mu_2$ equals $\mu_{(m)}$ ($\mu_{(m)}$ is as defined in problem $\mathcal{C}$) while the other one is zero. Suppose $Q=Z_{i_1}-Z_{i_2}$. This classification problem is equivalent to testing the following problem:
    $$H_0 : Q \sim N(-\mu_{(m)}, 2(1-\rho)) \quad vs \quad H_1: Q \sim N(\mu_{(m)}, 2(1-\rho)), \quad \mu_{m} >0.$$
\end{enumerate}}
\noindent We focus on problem $\mathcal{C}_{2}$ at first. The optimal stopping rule for problem $\mathcal{C}_{2}$ is given by $SPRT_{\alpha \wedge \beta}\bigg(-\mu_{(m)}, \mu_{(m)}, 2(1-\rho)\bigg)$. The expansion for its ASN, using \autoref{SPRToptimality}, is given by 
\begin{equation}\mathbb{E}\bigg[T_{SPRT_{\alpha \wedge \beta}\big(-\mu_{(m)}, \mu_{(m)}, 2(1-\rho)\big)}\bigg] \sim \frac{1-\rho}{\mu_{(m)}^2} \cdot  |\log (\alpha \wedge \beta)|. \label{problemc2} \end{equation}
\textcolor{black}{
We focus on problem $\mathcal{C}_{1}$ now. Suppose $T(\mathcal{C}_{1})$ denotes the stopping time for the gap rule for $\mathcal{C}_{1}$. We note that
$$R_{(m),n}-R_{(m+1),n} \geq \min_{i \in \mathcal{A}} R_{i,n}-\max_{j \in \mathcal{A}^c} R_{j,n}.$$
This is because, when $\displaystyle \min_{i \in \mathcal{A}} R_{i,n}-\max_{j \in \mathcal{A}^c} R_{j,n} \geq 0$, $R_{(m),n}-R_{(m+1),n} = \displaystyle \min_{i \in \mathcal{A}} R_{i,n}-\max_{j \in \mathcal{A}^c} R_{j,n}$. Otherwise, $R_{(m),n}-R_{(m+1),n}\geq 0 > \displaystyle \min_{i \in \mathcal{A}} R_{i,n}-\max_{j \in \mathcal{A}^c} R_{j,n}$. So, 
\begin{align*}
    T(\mathcal{C}_{1}) = \inf_{n} \bigg\{  R_{(m),n}-R_{(m+1),n}  \geq G\bigg\}
    & \leq  \inf_{n} \bigg\{ \min_{i \in \mathcal{A}} R_{i,n}-\max_{j \in \mathcal{A}^c} R_{j,n} \geq G\bigg\} \\
    & =  \inf_{n} \bigg\{ \min_{i \in \mathcal{A}, j \in \mathcal{A}^{c}} \big(R_{i,n}-R_{j,n} \big)\geq G\bigg\}.
\end{align*}
\autoref{sprtoptimalitynew} gives, as $\alpha, \beta \to 0$,
\begin{equation}\label{problemc1}
    \frac{\mathbb{E}[T(\mathcal{C}_1)]}{|\log (\alpha \wedge \beta)|} \leq \frac{1-\rho}{\mu_{(m)}^2}.
\end{equation}}
We focus on problem $\mathcal{C}$ now. Evidently, the average sample number required for problem $\mathcal{C}$ lies between the average sample numbers required for problem $\mathcal{C}_1$ and $\mathcal{C}_2$. Mathematically, 
\begin{align*}
    & ASN(\mathcal{C}_2) \leq ASN(\mathcal{C}) \leq ASN(\mathcal{C}_1) \\
    \text{i.e,} \hspace{2mm} & \mathbb{E}\bigg[T_{SPRT_{\alpha \wedge \beta}\big(-\mu_{(m)}, \mu_{(m)}, 2(1-\rho)\big)}\bigg] \leq \mathbb{E}_{\mathcal{A}}\left[T_{Gap}^{\star}\right] \leq \mathbb{E}[T(\mathcal{C}_{1})].
\end{align*}
The first inequality above is valid because $\mathcal{C}_{2}$ requires only one decision. We justify the first inequality along the following lines. If the first inequality does not hold, then in order to solve $\mathcal{C}_2$, one can think of additional $K-2$ hypothetical streams in which exactly $m-1$ streams have positive means, and can have a stopping rule with less ASN. But then this would imply that $SPRT_{\alpha \wedge \beta}\big(-\mu_{(m)}, \mu_{(m)}, 2(1-\rho)\big)$ is not the optimal stopping rule for problem $\mathcal{C}_2$. Hence, contradiction! So, the first inequality is valid. 

The second inequality holds because in $\mathcal{C}_1$ the positive $\mu$'s are much closer to $0$ than in $\mathcal{C}$. Equations \eqref{problemc2} and \eqref{problemc1} give, as $\alpha, \beta \to 0$,
\begin{equation}\label{problemc}\mathbb{E}_{\mathcal{A}}\left[T_{Gap}^{\star}\right] \sim \frac{1-\rho}{\mu_{(m)}^2}\cdot|\log (\alpha \wedge \beta)|.\end{equation}
We now establish that our proposed gap rule has the asymptotically optimal expected sample size. Towards this, we note that the minimal expected sample size to find $m$ signals among $K$ streams is not less than the minimal expected sample size to find $1$ signal among $2$ streams. 

\noindent Mathematically, as $\alpha, \beta \to 0$,
\begin{align*}
    \inf _{(T, d) \in \Delta_{\alpha, \beta}^{FWER}\left(\mathcal{P}_{m}\right)} \mathbb{E}_{\mathcal{A}}[T] \geq  \inf _{(T, d) \in \Delta_{\alpha, \beta}^{FWER}\left(\mathcal{P}_{1}\right)} \mathbb{E}_{\mathcal{A}}[T] 
    = & \mathbb{E}\bigg[T_{SPRT_{\alpha \wedge \beta}\big(-\mu_{(m)}, \mu_{(m)}, 2(1-\rho)\big)}\bigg] \quad \text{(since SPRT is the optimal test here)}\\
    \sim & \frac{1-\rho}{\mu_{(m)}^2} \cdot  |\log (\alpha \wedge \beta)| \quad \text{(from equation \eqref{problemc2})} \\
\sim & \mathbb{E}_{\mathcal{A}}\left[T_{Gap}^{\star}\right] \quad \text{(from equation \eqref{problemc})}.
\end{align*}
\noindent Thus, we obtain the following:
\begin{theorem}\label{1stoptimality}
    For every $\mathcal{A} \in \mathcal{P}_{m}$, we have as $\alpha, \beta \rightarrow 0$
$$\mathbb{E}_{\mathcal{A}}\left[T_{Gap}^{\star}\right] \sim \frac{1-\rho}{\mu_{(m)}^2}\cdot|\log (\alpha \wedge \beta)| \sim \inf _{(T, d) \in \Delta_{\alpha, \beta}^{FWER}\left(\mathcal{P}_{m}\right)} \mathbb{E}_{\mathcal{A}}[T] .$$
\end{theorem}
\subsection{Proposed procedure when upper and lower bounds on the number of signals are available}
In this as well as the next section, we assume that all the alternative $\mu_{i}$'s are equal to $\mu$. Here we consider the setup in which it is known beforehand that there are at least $\ell$ and at most $u$ signals (both exclusive) for some $0 < l \leq u < K$. This corresponds to considering the class $\mathcal{P}_{\ell, u}$. We propose the gap rule
$$\begin{aligned}
T_{Gap}^{\star \star}(l,u,e_{\textcolor{black}{n}}) & :=\inf \left\{n \geq 1: \max_{\ell < i < u} \left(S_{(i),n} - S_{(i+1),n}\right) \geq e_{\textcolor{black}{n}}\right\}
\end{aligned}$$
where $e_{\textcolor{black}{n}}$ is a suitably defined \textcolor{black}{time-varying threshold}. Let $p$ be the index where the above maximum occurs at time $T_{Gap}^{\star \star}(l,u,e_{\textcolor{black}{n}})$. The set of rejected hypotheses is given by
 $d_{Gap}^{\star \star} := \left\{i_{1}(T_{Gap}^{\star \star}),\ldots,i_{p}(T_{Gap}^{\star \star})\right\}$, where $S_{(t),n} = S_{i_t,n}$. Now, 
\begin{align*}
    FWER_{I}(T_{Gap}^{\star \star}, d_{Gap}^{\star \star})
    = \mathbb{P}_{\mathcal{A}}\left(\bigcup_{i \notin \mathcal{A}} \{ d_{Gap,i}^{\star \star} = 1\}\right)
    \leq & \sum_{i \notin \mathcal{A}} \mathbb{P}_{\mathcal{A}} \left(d_{Gap,i}^{\star \star} = 1\right) \\
    \leq & (K-l) \mathbb{P}_{\mathcal{A}} \left(d_{Gap,i}^{\star \star} = 1\right) \quad \text{(for any $i \notin \mathcal{A}$)} \\
    \leq & 2(K-l) \mathbb{P}_{\mathcal{A}}\left(\bigcup_{j \notin \mathcal{A}, j\neq i} \{S_{i,n} - S_{j,n} \geq e_{\textcolor{black}{n}}\}\right) \\
    \leq & 2(K-l)(K-l-1) \mathbb{P}_{\mathcal{A}}(S_{i,n} - S_{j,n} \geq e_{\textcolor{black}{n}})  \quad \text{(for $j \notin \mathcal{A}, j \neq i$)} \\
    = & 2(K-l)(K-l-1) \Phi \left( \frac{-e_{\textcolor{black}{n}}}{\sqrt{2(1-\rho)n}}\right) \hspace{2mm}(\Phi \hspace{2mm} \text{is the cdf of} \hspace{2mm}N(0,1))\\
    = & 2(K-l)(K-l-1) \mathbb{P}_{\mathcal{A}}\left(S_{i,n} \geq \frac{e_{\textcolor{black}{n}}}{\sqrt{2}}\right).
\end{align*}
Hence, the chosen $e_{\textcolor{black}{n}}$ should satisfy
\begin{align*}
    & 2(K-l)(K-l-1) \mathbb{P}_{\mathcal{A}}\left(S_{i,n} \geq \frac{e_{\textcolor{black}{n}}}{\sqrt{2}}\right)  \leq \alpha \\
    \implies & \mathbb{P}_{\mathcal{A}}\left(S_i \geq \frac{e_{\textcolor{black}{n}}}{\sqrt{2}}\right) \leq \frac{\alpha}{2(K-l)(K-l-1)}.
\end{align*}
Thus, we can choose $e_{\textcolor{black}{n}}$ to be
$$\frac{e_{\textcolor{black}{n}}}{\sqrt{2}} = \frac{1-\rho}{\mu}\cdot |\log \left(\frac{\alpha}{2(K-l)(K-l-1)}\right) | + n\mu/2.$$
This is because this is the cutoff used by $SPRT_{(\frac{\alpha}{2(K-l)(K-l-1)}, 0)}(0,\mu, 1-\rho)$. 

\noindent One can also show the following exactly as in the same way as above:
\begin{align*}
     FWER_{II}(T_{Gap}^{\star \star}, d_{Gap}^{\star \star}) & \leq 2u(u-1) \mathbb{P}_{\mathcal{A}}(S_i - S_j \geq e_{\textcolor{black}{n}})  \quad \text{(for $i, j \in \mathcal{A}, j\neq i$)} \\
    & = 2u(u-1) \mathbb{P}_{\mathcal{A}}\left(S_i \geq \frac{e_{\textcolor{black}{n}}}{\sqrt{2}}\right). 
\end{align*}
Hence, the chosen $e_{\textcolor{black}{n}}$ should satisfy
\begin{align*}
& \mathbb{P}_{\mathcal{A}}\left(S_i \geq \frac{e_{\textcolor{black}{n}}}{\sqrt{2}}\right) \leq \frac{\beta}{2u(u-1)}.
\end{align*}
This leads to the choice
$$\frac{e_{\textcolor{black}{n}}}{\sqrt{2}} = \frac{1-\rho}{\mu}\cdot |\log \left(\frac{\beta}{2u(u-1)}\right) | + n\mu/2.$$
Therefore, we choose the following $e_{\textcolor{black}{n}}$:
$$e_{\textcolor{black}{n}} = \frac{1-\rho}{\mu}\cdot \max \left \{ |\log \left(\frac{\alpha}{2(K-l)(K-l-1)}\right) |, |\log \left(\frac{\beta}{2u(u-1)}\right) |\right\} + n\mu/2.$$

\noindent The previous derivations result in the following:
\begin{align*}
    & FWER_{I}(T_{Gap}^{\star \star}, d_{Gap}^{\star \star}) \leq 2(K-l)(K-l-1)PICS \Big[SPRT_{\left(\frac{\alpha}{2(K-l)(K-l-1)},0\right)}(0,\mu, 1-\rho) \Big], \\
    & FWER_{II}(T_{Gap}^{\star \star}, d_{Gap}^{\star \star}) \leq 2u(u-1)PICS \Big[ SPRT_{\left(\frac{\beta}{2u(u-1)},0\right)}(0,\mu, 1-\rho) \Big]. 
\end{align*} 
However, those derivations also give the following inequalities:
\begin{align*}
    & FWER_{I}(T_{Gap}^{\star \star}, d_{Gap}^{\star \star}) \geq PICS \Big[SPRT_{\left(\frac{\alpha}{2(K-l)(K-l-1)},0\right)}(0,\mu, 1-\rho) \Big], \\
    & FWER_{II}(T_{Gap}^{\star \star}, d_{Gap}^{\star \star}) \geq PICS \Big[ SPRT_{\left(\frac{\beta}{2u(u-1)},0\right)}(0,\mu, 1-\rho) \Big]. 
\end{align*} 
Combining all these four inequalities, we obtain, 
$$\mathbb{E}_{\mathcal{A}}(T_{Gap}^{\star \star}) = \max \left \{ \mathbb{E}\left(T_{SPRT_{\left(\frac{\alpha}{2(K-l)(K-l-1)},0\right)}(0,\mu, 1-\rho)} \right) , \mathbb{E}\left(T_{SPRT_{\left(\frac{\beta}{2u(u-1)},0\right)}(0,\mu, 1-\rho)} \right) \right \}.$$
This gives, as $\alpha, \beta \to 0$, 
$$
\mathbb{E}_{\mathcal{A}}\left[T_{Gap}^{\star \star}\right] \sim \frac{2(1-\rho)}{\mu^2}\cdot|\log (\alpha \wedge \beta)|.
$$
\textcolor{black}{
We now establish that our proposed gap rule has the asymptotically optimal expected sample size. Towards this, we note that the minimal expected sample size to find  signals among $K$ streams is not less than the minimal expected sample size to decide whether a particular stream has signal. Mathematically, as $\alpha, \beta \to 0$,
\begin{align*}
    \inf _{(T, d) \in \Delta_{\alpha, \beta}^{FWER}\left(\mathcal{P}_{l,u}\right)} \mathbb{E}_{\mathcal{A}}[T] \geq & \inf _{(T, d) \in \Delta_{\alpha, \beta}^{FWER}\left(\mathcal{P}_{0,1}\right)} \mathbb{E}_{\mathcal{A}}[T] \\
    = & \mathbb{E}\bigg[T_{SPRT_{\alpha \wedge \beta}\big(0, \mu, (1-\rho)\big)}\bigg] \quad \text{(since SPRT is the optimal test here)}\\
    \sim & \frac{2(1-\rho)}{\mu^2} \cdot  |\log (\alpha \wedge \beta)| \quad \text{(from equation \eqref{problemc2})} \\
\sim & \mathbb{E}_{\mathcal{A}}\left[T_{Gap}^{\star \star}\right] \quad \text{(from the above equation)}.
\end{align*}}
\noindent Thus, we obtain the following:
\begin{theorem}\label{2ndoptimality}
    For every $\mathcal{A} \in \mathcal{P}_{\ell, u}$, we have as $\alpha, \beta \rightarrow 0$,
$$
\mathbb{E}_{\mathcal{A}}\left[T_{Gap}^{\star \star}\right] \sim \frac{2(1-\rho)}{\mu^2}\cdot|\log (\alpha \wedge \beta)| \sim \inf _{(T, d) \in \Delta_{\alpha, \beta}^{FWER}\left(\mathcal{P}_{\ell, u}\right)} \mathbb{E}_{\mathcal{A}}[T] .
$$
\end{theorem}

\section{General Error Rate controlling Procedures\label{sec:chap6sec5}}
As mentioned earlier, \cite{heBartroff2021} studied asymptotic optimality of general multiple testing error metrics under the independent streams framework. In this section, we focus on deriving similar results under dependence. Denote the type 1 and 2 versions of a generic multiple testing error metric by $\mathrm{MTE}=\left(\mathrm{MTE}_1, \mathrm{MTE}_2\right)$, which is any pair of functions mapping MTPs into $[0,1]$. We mention some widely used multiple testing error metrics below.

For any MTP under consideration, let $V, W$ and $R$ respectively denote the number of true null hypotheses rejected, the number of false null hypotheses accepted, and the total number of null hypotheses rejected. The type I and type II familywise error probabilities are given by
$$
FWER_{1, \mathcal{A}}=\mathbb{P}_{\mathcal{A}}(V \geq 1), \quad FWER_{2, \mathcal{A}}=\mathbb{P}_{\mathcal{A}}(W \geq 1).
$$
In this and other MTEs, as in \cite{heBartroff2021}, we will include the procedure being evaluated as
an argument (e.g., $FWER_{1,\mathcal{A}}(T , d)$) when needed. However, we shall omit it when it is clear from the context or when a statement holds
for any procedure. Similarly, we shall omit the subscript $\mathcal{A}$ in expressions that are valid for arbitrary signal sets.

The false discovery and non-discovery rates (FDR, FNR) are given by
$$
FDR=\mathbb{E}\left(\frac{V}{R \vee 1}\right), \quad FNR=\mathbb{E}\left(\frac{W}{(K-R) \vee 1}\right).
$$

\cite{Storey2003} considers the positive false discovery rate and positive false non-discovery rate:
$$pFDR=E\left(\frac{V}{R} \mid R \geq 1\right), \quad pFNR=E\left(\frac{W}{K-R} \mid J-R \geq 1\right).$$

For a generic metric MTE, let
$$
\Delta_{\alpha, \beta}^{\mathrm{MTE}}(\mathcal{P})=\left\{(T, d): \mathrm{MTE}_{1, \mathcal{A}} \leq \alpha \text { and } \mathrm{MTE}_{2, \mathcal{A}} \leq \beta \text { for all } \mathcal{A} \in \mathcal{P}\right\}.
$$

We now mention a result on asymptotic optimality of genral error metric controlling procedures. 

\begin{theorem}\label{theoremongenerror}
Consider the equicorrelated streams setup with common correlation $\rho >0$. Fix $1 \leq m \leq K-1$ and let $\left(T_{Gap}^{\star}(c), d_{Gap}^{\star}(c)\right)$ denote our gap rule with number of signals $m$ and threshold $c>0$. Let MTE be a multiple testing error metric such that:

(i) there is a constant $C_1$ such that
\begin{equation}\label{geneq1}
\operatorname{MTE}_{i, \mathcal{A}}\left(T_{Gap}^{\star}(c), d_{Gap}^{\star}(c)\right) \leq C_1 \cdot FWER_{i, \mathcal{A}}\left(T_{Gap}^{\star}(c), d_{Gap}^{\star}(c)\right)
\end{equation}
for $i=1$ and 2 , for all $\mathcal{A} \in \mathcal{P}_m$, and for all $c>0$, and

(ii) there is a constant $C_2$ such that
\begin{equation}\label{geneq2}
\operatorname{MTE}_{i, \mathcal{A}}(T, d) \geq C_2 \cdot FWER_{i, \mathcal{A}}(T, d)
\end{equation}
for $i=1$ and 2 , for all $\mathcal{A} \in \mathcal{P}_m$, and for all procedures $(T, d)$.

Given $\alpha, \beta \in(0,1),$ let $\left(T_{Gap}^{\prime}, d_{Gap}^{\prime}\right)$ denote our proposed gap rule with number of signals $m$ and threshold
$$
c=\left|\log \left(\left(\alpha / C_1\right) \wedge\left(\beta / C_1\right)\right)\right|+\log (m(K-m)) .
$$
Then the following hold.

(1) The procedure $\left(T_{Gap}^{\prime}, d_{Gap}^{\prime}\right)$ is valid for MTE control. That is,
\begin{equation}\label{geneq3}
\left(T_{Gap}^{\prime}, d_{Gap}^{\prime}\right) \in \Delta_{\alpha, \beta}^{\mathrm{MTE}}(\mathcal{P}_m) .
\end{equation}

(2) The procedure $\left(T_{Gap}^{\prime}, d_{Gap}^{\prime}\right)$ is asymptotically optimal for MTE control with respect to class $\mathcal{P}_m$. That is, for all $\mathcal{A} \in \mathcal{P}_m$,
\begin{equation}\label{geneq4}
 \mathbb{E}_{\mathcal{A}}\left(T_{Gap}^{\prime}\right) \sim \frac{1-\rho}{\mu^2}\cdot |\log (\alpha \wedge \beta)| \sim \inf _{(T, d) \in \Delta_{\alpha, \beta}^{\mathrm{MTE}}(\mathcal{P}_m)} \mathbb{E}_\mathcal{A}(T)
\end{equation}
as $\alpha, \beta \rightarrow 0$.
\end{theorem}

\begin{proof}
For the first part, we fix arbitrary $\mathcal{A} \in \mathcal{P}_m$. We have
$$
\operatorname{FWER}_{1, A}\left(T_{Gap}^{\prime}, d_{Gap}^{\prime}\right) \leq \alpha / C_1 \text { and } \operatorname{FWER}_{2, A}\left(T_{Gap}^{\prime}, d_{Gap}^{\prime}\right) \leq \beta / C_1 \text {. }
$$
Applying \eqref{geneq1} yields
$$
\operatorname{MTE}_{1, A}\left(T_{Gap}^{\prime}, d_{Gap}^{\prime}\right)\leq \alpha \quad \text{and} \quad \operatorname{MTE}_{2, A}\left(T_{Gap}^{\prime}, d_{Gap}^{\prime}\right) \leq \beta
$$
Hence \eqref{geneq3} is established. For the second part, again we fix arbitrary $\mathcal{A} \in \mathcal{P}_m$ and consider $\alpha, \beta \rightarrow 0$. We have, as $\alpha, \beta \rightarrow 0$, 
\begin{equation}\label{geneq5}
\mathbb{E}_{\mathcal{A}}\left(T_{Gap}^{\prime}\right) \sim \frac{1-\rho}{\mu^2}\cdot |\log (\alpha \wedge \beta)|.
\end{equation}
It remains to show that the r.h.s in the above equation is also a lower bound for any procedure in $\Delta_{\alpha, \beta}^{\mathrm{MTE}}(\mathcal{P}_m)$. \eqref{geneq2} gives
$$
\Delta_{\alpha, \beta}^{\mathrm{MTE}}(\mathcal{P}_m) \subseteq \Delta_{\alpha/C_2, \beta/C_2}^{\mathrm{FWER}}(\mathcal{P}_m).
$$
This implies
$$
\inf _{(T, d) \in \Delta_{\alpha, \beta}^{\mathrm{MTE}}(\mathcal{P}_m)} \mathbb{E}_{\mathcal{A}}(T) \geq \inf _{(T, d) \in \Delta_{\alpha/C_2, \beta/C_2}^{\mathrm{FWER}}(\mathcal{P}_m)} \mathbb{E}_{\mathcal{A}}(T).
$$
The latter is of the order $\frac{1-\rho}{\mu^2}\cdot |\log (\alpha \wedge \beta)|$. This, combined with \eqref{geneq5} gives the desired result. 
\end{proof}

\begin{remark}
$(FDR, FNR)$ and $(pFDR, pFNR)$ satisfy the conditions mentioned 
in \autoref{theoremongenerror}. This is because of the following inequalities \citep{heBartroff2021}:
    $$\frac{1}{K} \cdot FWER_1\leq FDR \leq FWER_1, \quad \frac{1}{K} \cdot FWER_2 \leq FNR \leq FWER_2,$$
$$\frac{1}{K} \cdot FWER_1\leq pFDR, \quad \frac{1}{K} \cdot FWER_2 \leq pFNR \leq FWER_2,$$
$$pFDR(T_{Gap}^{\star}) \leq FWER_1(T_{Gap}^{\star}), \quad pFNR(T_{Gap}^{\star}) \leq FWER_2(T_{Gap}^{\star}).$$
\end{remark}

 We also have the following similar asymptotic optimality result for general error metrics in the case when lower and upper bounds for the number of signals is available:

 \begin{theorem}\label{theoremongenerror2}
Consider the equicorrelated streams setup with common correlation $\rho >0$. Fix integers $0 < \ell<u < K$ and let $T_{Gap}^{\star \star}(e)$ denote the gap-intersection rule with strict bounds $\ell$, $u$ on the number of signals and threshold $e$. Let MTE be multiple testing error metric such that:

(i) there is a constant $C_1$ such that
\begin{equation}\label{geneq6}
\operatorname{MTE}_{i, \mathcal{A}}\left(T_{Gap}^{\star \star}(e), d_{Gap}^{\star \star }(e)\right) \leq C_1 \cdot FWER_{i, \mathcal{A}}\left(T_{Gap}^{\star}(e), d_{Gap}^{\star \star}(e)\right)
\end{equation}
for $i=1$ and 2 , for all $\mathcal{A} \in \mathcal{P}_{\ell, u}$, and for all $e>0$, and

(ii) there is a constant $C_2$ such that
\begin{equation}\label{geneq7}
\operatorname{MTE}_{i, \mathcal{A}}(T, d) \geq C_2 \cdot FWER_{i, \mathcal{A}}(T, d)
\end{equation}
for $i=1$ and 2 , for all $\mathcal{A} \in \mathcal{P}_{\ell, u}$, and for all procedures $(T, d)$.

Given $\alpha, \beta \in(0,1),$ let $\left(T_{Gap}^{\prime \prime}, d_{Gap}^{\prime}\right)$ denote our proposed gap rule with bounds $\ell, u$ on the number of signals and threshold
$$e = \frac{1-\rho}{\mu}\cdot \max \left \{ |\log \left(\frac{\alpha/ C_1}{2(K-l)(K-l-1)}\right) |, |\log \left(\frac{\beta/C_1}{2u(u-1)}\right) |\right\} + n\mu/2.$$
Then the following hold.

(1) The procedure $\left(T_{Gap}^{\prime}, d_{Gap}^{\prime}\right)$ is valid for MTE control. That is,
\begin{equation}\label{geneq8}
\left(T_{Gap}^{\prime \prime}, d_{Gap}^{\prime \prime }\right) \in \Delta_{\alpha, \beta}^{\mathrm{MTE}}(\mathcal{P}_{\ell, u}) .
\end{equation}

(2) The procedure $\left(T_{Gap}^{\prime \prime}, d_{Gap}^{\prime \prime }\right)$ is asymptotically optimal for MTE control with respect to class $\mathcal{P}_{\ell, u}$. That is, for all $\mathcal{A} \in \mathcal{P}_{\ell, u}$,
\begin{equation}\label{geneq9}
 \mathbb{E}_{\mathcal{A}}\left(T_{Gap}^{\prime \prime}\right) \sim \frac{2(1-\rho)}{\mu^2}\cdot |\log (\alpha \wedge \beta)| \sim \inf _{(T, d) \in \Delta_{\alpha, \beta}^{\mathrm{MTE}}(\mathcal{P}_{\ell, u})} \mathbb{E}_\mathcal{A}(T)
\end{equation}
as $\alpha, \beta \rightarrow 0$.
\end{theorem}
The proof is exactly similar to the preceding and hence omitted. 

\section{The case with unknown mean and equicorrelation\label{sec:chap6sec6}}
Throughout this work, we have assumed that both the mean (i.e, the signal strength) and the common correlation are known. In this section, we extend our results to the case when $\mu$ and $\rho$ are unknown. For technical simplicity, we consider the known $m$ case with common positive mean $\mu$. 

\noindent We now define some random variables to construct consistent estimators of $\mu$ and $\rho$. Towards this, For any $j \in \mathbb{N}$, let
$$A_j := \frac{1}{m} \sum_{i=1}^{K}  X_{ij}, \quad B_j := \frac{1}{K(K-1)} \sum_{i\neq i^{\prime}} X_{ij}X_{i^{\prime}j}.$$
$A_{j}$'s are i.i.d random variables with mean $\mu$ and finite variance. This implies $\hat{\mu}_n:= \frac{1}{n} \sum_{j=1}^n A_j$ is a consistent estimator of $\mu$. Also, the strong law of large numbers (applied on the random variables $A_j$) gives that $\hat{\mu}_n$ converges to $\mu$ almost surely.

Similarly, $B_{j}$'s are i.i.d random variables with mean $\mathbb{E}(B_j) = \rho + \frac{m(m-1)}{K(K-1)}\mu^2$ and finite variance. 
Hence, the random variable $\frac{1}{n}\sum_{j=1}^{n} B_j$ is a consistent estimator for $\rho + \frac{m(m-1)}{K(K-1)}\mu^2$. Using the consistency of $\hat{\mu}_n$, $\frac{1}{n}\sum_{j=1}^{n} B_j$ and applying the continuous mapping theorem, we obtain that the estimator 
$$\hat{\rho}_n:= \frac{1}{n}\sum_{j=1}^{n} B_j - \frac{m(m-1)}{K(K-1)} {\hat{\mu}_n}^2$$
is consistent for $\rho$. We recall the definition of our proposed stopping time when $\mu$ and $\rho$ are known:
$$\begin{aligned}
T_{Gap}^{\star}(m,c) := & \inf \left\{n \geq 1:  R_{(m),n} - R
_{(m+1),n} \geq \frac{1-\rho}{\mu} \cdot c\right\}
\end{aligned}$$
where $c:=|\log (\alpha \wedge \beta)|+\log (m(K-m))$. Mimicking this rule, we define a stopping time for the unknown mean and unknown correlation case: 
$$\begin{aligned}
\widehat{T}_{Gap}^{\star}(m,c) := & \inf \left\{n \geq M:  R_{(m),n} - R
_{(m+1),n} \geq \frac{1-\hat{\rho}_n}{\hat{\mu}_n} \cdot c\right\}
\end{aligned}$$
where $M$ is some fixed large positive integer. 

Fix $\epsilon>0$. Define two random variables:
$$\begin{aligned}
W_1 := & \inf \left\{n \geq 1:  R_{(m),n} - R
_{(m+1),n} \geq \bigg[\frac{1-\rho}{\mu}-\epsilon\bigg] \cdot c\right\}, \\
W_2 := & \inf \left\{n \geq 1:  R_{(m),n} - R
_{(m+1),n} \geq \bigg[\frac{1-\rho}{\mu}+\epsilon\bigg] \cdot c\right\}.
\end{aligned}$$
The function $f$ defined as $f(x,y)=\frac{1-y}{x}$ is continuous in $\mathbb{R}^{+}\times(0,1)$. The continuous mapping theorem implies that $\frac{1-\hat{\rho}_n}{\hat{\mu}_n}$ converges to $\frac{1-\rho}{\mu}$ almost surely. This gives 
\begin{align*}
    & W_1 \leq \widehat{T}_{Gap}^{\star}(m,c) \leq W_2 \quad \text{almost surely} \\
    \implies & \mathbb{E}[W_1] \leq \mathbb{E}[\widehat{T}_{Gap}^{\star}(m,c) ]\leq \mathbb{E}[W_2].& 
\end{align*}
Using \autoref{1stoptimality} we obtain
\begin{align*}
    & \frac{\frac{1-\rho}{\mu}-\epsilon}{\mu} \leq \lim_{\alpha, \beta \to 0} \frac{\mathbb{E}[\widehat{T}_{Gap}^{\star}(m,c) ]}{|\log(\alpha \wedge \beta)|}\leq 
    \frac{\frac{1-\rho}{\mu}+\epsilon}{\mu}.
\end{align*}
Since the above inequality holds for arbitrary $\epsilon >0$, we have
\begin{equation}\lim_{\alpha, \beta \to 0} \frac{\mathbb{E}[\widehat{T}_{Gap}^{\star}(m,c) ]}{|\log(\alpha \wedge \beta)|} = \frac{1-\rho}{\mu^2}.\label{expression}\end{equation}
Thus, in the known $m$ case with common positive mean $\mu$, we obtain 
$$\lim_{\alpha, \beta \to 0} \frac{\mathbb{E}[\widehat{T}_{Gap}^{\star}(m,c) ]}{\mathbb{E}[T_{Gap}^{\star}(m,c) ]} = 1,$$ supporting the asymptotic optimality of $\widehat{T}_{Gap}^{\star}(m,c)$. However, from the definition of $\widehat{T}_{Gap}^{\star}(m,c)$ and \autoref{1stoptimality}, it is clear that 
$$\lim_{\alpha, \beta \to 0} \frac{\mathbb{E}[\widehat{T}_{Gap}^{\star}(m,c) ]}{|\log(\hat{\alpha} \wedge \hat{\beta})|} = \frac{1-\rho}{\mu^2}$$
where $\hat{\alpha}, \hat{\beta}$ are the levels at which $FWER_{I}$ and $FWER_{II}$ are respectively controlled by $\widehat{T}_{Gap}^{\star}(m,c)$. Comparing with \eqref{expression}, we get
$$\lim_{\alpha, \beta \to 0} \frac{|\log(\hat{\alpha} \wedge \hat{\beta})|}{|\log(\alpha \wedge \beta)|} = 1.$$We summarize the findings of this section in the following result:
\begin{theorem}
 Consider the known $m$ case with unknown signal strength $\mu$ and correlation $\rho$. Let $\hat{\mu}_n, \hat{\rho}_n, \widehat{T}_{Gap}^{\star}(m,c)$ be as defined as in this section. We have the following:
 \begin{enumerate}
     \item Suppose $\hat{\alpha}, \hat{\beta}$ are the levels at which $FWER_{I}$ and $FWER_{II}$ are respectively controlled by $\widehat{T}_{Gap}^{\star}(m,c)$. Then, 
     $$\lim_{\alpha, \beta \to 0} \frac{|\log(\hat{\alpha} \wedge \hat{\beta})|}{|\log(\alpha \wedge \beta)|} = 1.$$
     \item $$\lim_{\alpha, \beta \to 0} \frac{\mathbb{E}[\widehat{T}_{Gap}^{\star}(m,c) ]}{\mathbb{E}[T_{Gap}^{\star}(m,c) ]} = 1.$$
 \end{enumerate}
\end{theorem}
\section{Simulation Study\label{sec:chap6sec7}}

In this section, we empirically elucidate the performance of our proposed gap rule. We consider the known $m$ case. The set-up is the following:
\begin{enumerate}[label=(\roman*)]
    \item Number of streams: $K=20$.
    \item Number of streams in which signal is present: $m=3$.
    \item Means under the alternative hypotheses: $\mu_{i}=.5$ for each $i \in \mathcal{A}$.
    \item Common correlation among the streams: we consider three values of $\rho$, namely .1, .5 and .9. 
    \item Desired level of familywise error control: we take $\alpha=\beta$ for simplicity. We have considered eight values of $\alpha$, for which $-\log_{10}(\alpha) \in 
    \{2,4,6,8,10,12,14,16\}.$ This means $\alpha$ ranges from $.01$ to $10^{-16}$.
\end{enumerate}

\noindent \autoref{fig:1} portrays the estimates of the average sample numbers (ASN) for $\rho =.1, .5, .9$ (in solid, dashed and dotted lines respectively) under different values of $\alpha$. We see that the ASNs decrease with $\rho$. Also, the ASNs are nearly linear in $-\log_{10}(\alpha)$, as also seen in \autoref{1stoptimality}. 

\begin{figure}[ht]
     \centering
     \includegraphics[width=10cm]{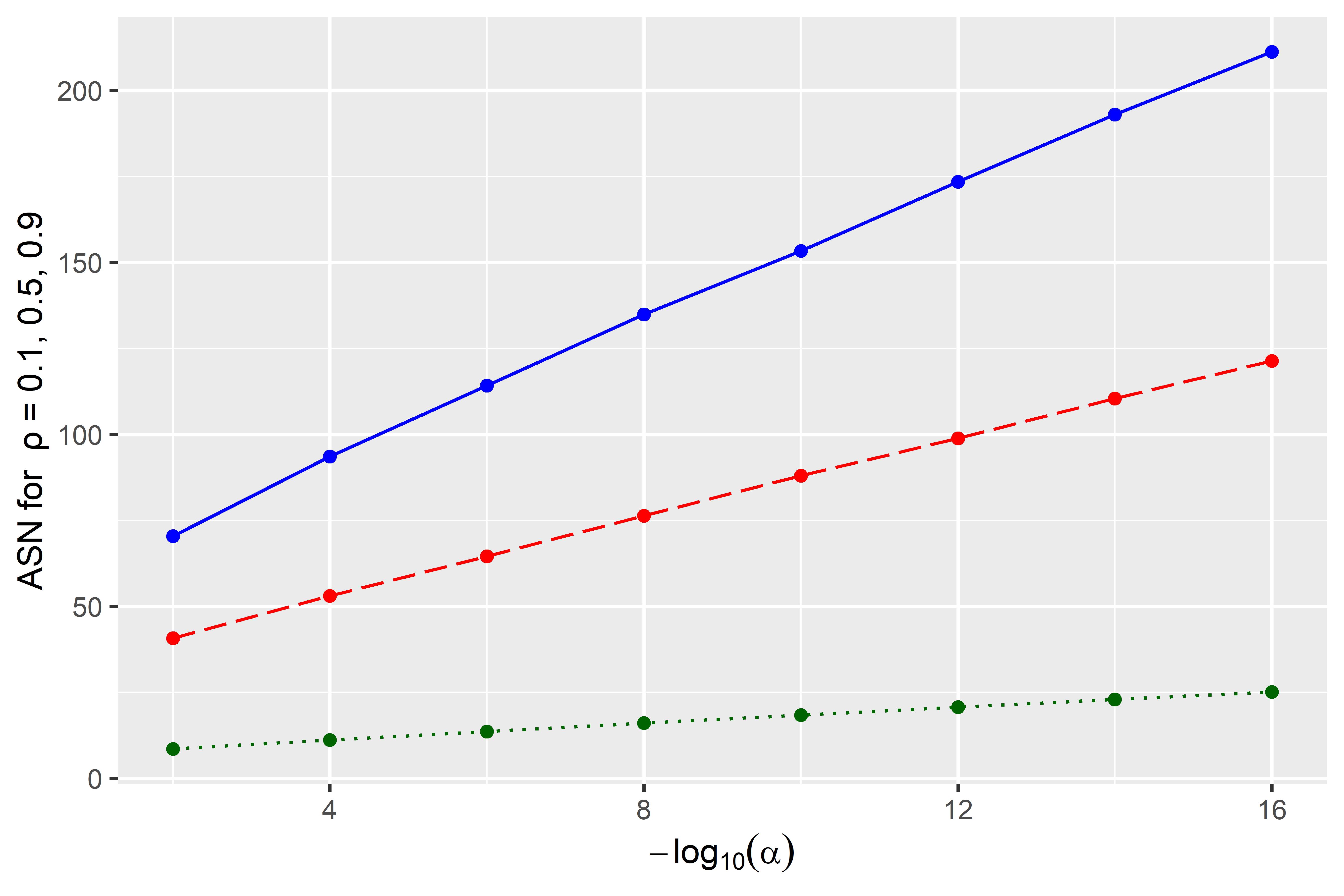}
  \caption{Estimates of average sample numbers of our proposed gap rule for different combinations of $(\rho, \alpha)$. Based on 1000 repetitions.}
  \label{fig:1}
\end{figure}  

\begin{figure}[!ht]
     \centering
     \includegraphics[width=10cm]{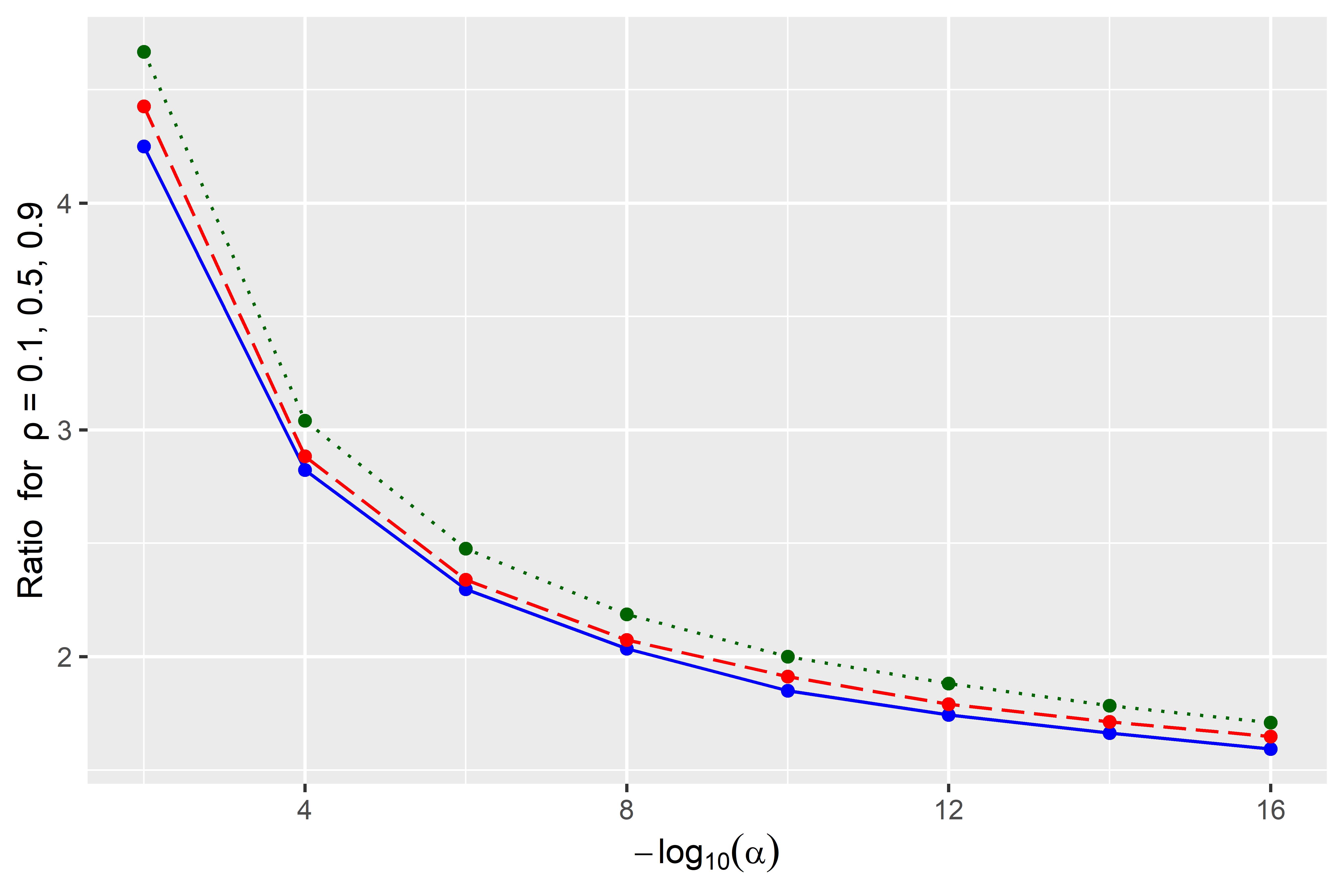}
  \caption{Ratios of the ASNs and the theoretical approximations of our proposed gap rule for different combinations of $(\rho, \alpha)$. Based on 1000 repetitions.}
  \label{fig:2}
\end{figure}

\autoref{fig:2} presents the ratios of the ASNs and the corresponding theoretical approximations (provided by \autoref{1stoptimality}) for $\rho =.1, .5, .9$ (in solid, dashed and dotted lines respectively) under different values of $\alpha$. Since the theoretical approximation is a lower bound to the ASN, the ratio is always greater than 1. We see that the ratio decreases to $1$ as $\alpha \to 0$. Also, the convergence to $1$ is slower for higher values of $\rho$.

\section{Concluding Remarks\label{sec:chap6sec8}}
Our results (e.g., Theorem \ref{1stoptimality} and Theorem \ref{2ndoptimality}) elucidate that the asymptotically optimal average sample numbers are decreasing in the common correlation $\rho$. A high value of $\rho$ implies a strong correlation among the streams, and so one might expect that it should require fewer observations on average to detect the signals. Our results illustrate this remarkable blessing of dependence. This role of correlation is in stark contrast to its role in the frequentist paradigm. Several popular and widely used procedures fail to hold the FWER at a positive level asymptotically under positively correlated Gaussian frameworks \citep{deyspl, deystpa}.  Thus, our results illuminate that dependence might be a blessing or a curse, subject to the type of dependence or the underlying paradigm.

\cite{FDR2007} remark that false discoveries are challenging to tackle in models with complex dependence structures, e.g., arbitrarily correlated Gaussian models. It would be interesting to explore if there are connections between the SPRT and the optimal sequential test rules under general dependencies. Throughout this work, we have considered multivariate Gaussian setup, frequent in various areas of stochastic modeling \citep{Hutchinson, olkin_viana_1995, Monhor_2011}. However, one interesting extension would be to study the signal detection problem under general distributions and to see whether similar connections with SPRT exist in those cases too.

\section*{Acknowledgements}

Dey sincerely acknowledges Prof. Thorsten Dickhaus for his constant support and encouragement throughout this work. 

\section*{Disclosure Statement}

The authors report there are no competing interests to declare.





\selectlanguage{english}
\bibliography{references}

\end{document}